\documentclass[11pt,a4paper]{article}
\usepackage{float}
\usepackage{graphics,epstopdf}
\usepackage{graphics,multicol}
\usepackage{graphicx}
\usepackage{ragged2e}
\usepackage{epsf,epsfig,amsfonts,amsgen,amsmath,amstext,amsbsy,amsopn,amsthm,amsfonts,amssymb,amscd
}
\usepackage{ebezier,eepic}
\usepackage{color}
\usepackage{multirow}
\usepackage{ragged2e}
\usepackage{listings}
\newtheorem{thm}{Theorem}

\newtheorem{lem}{Lemma}
\newtheorem{false statement}{False statement}

\theoremstyle{definition}

\newtheorem{claim}{Claim}

\newtheorem{conj}{Conjecture}

\newcounter{mathitem}
  {\begin{list}{{$(\roman{mathitem})$}}{
   \setcounter{mathitem}{0}
   \usecounter{mathitem}
   \setlength{\topsep}{0pt plus 2pt minus 0pt}
   \setlength{\parskip}{0pt plus 2pt minus 0pt}
   \setlength{\partopsep}{0pt plus 2pt minus 0pt}
   \setlength{\parsep}{0pt plus 2pt minus 0pt}
   \setlength{\leftmargin}{35pt}
   \setlength{\itemsep}{0pt plus 2pt minus 0pt}}}
  {\end{list}}

\begin{document}

\title{\bf\Large A Complete Solution to the Cvetkovi\'{c}-Rowlinson Conjecture}

\date{}
\author{
Huiqiu Lin\footnote{Department of Mathematics,
East China University of Science and Technology, Shanhai 200237, P.R. China.
Email: huiqiulin@126.com. Research supported by NSFC grant 11771141.}~~ and~~ Bo Ning\footnote{Corresponding author. College of Computer Science, Nankai University, Tianjin 300071, P.R. China.
Email: bo.ning@nankai.edu.cn. Research partially supported by NSFC grant 11971346.}}
\maketitle

\begin{abstract}
In 1990, Cvetkovi\'{c} and Rowlinson [The largest eigenvalue of a graph: a survey, Linear
Multilinear Algebra 28(1-2) (1990), 3--33] conjectured that among all outerplanar graphs
on $n$ vertices, $K_1\vee P_{n-1}$ attains the maximum spectral radius. In 2017, Tait and Tobin
[Three conjectures in extremal spectral graph theory, J. Combin. Theory, Ser. B
126 (2017) 137-161] confirmed the conjecture for sufficiently large values of $n$. In this
article, we show the conjecture is true for all $n\geq2$ except for $n=6$.
\end{abstract}



There is a long tradition of studying planar graphs. In particular, the study of spectral
radius of planar graphs is a fruitful topic in spectral graph theory and can be traced back
at least to Schwenk and Wilson \cite{SW78} who asked ``what can be said about the eigenvalues
of a planar graph?". In 1988, Hong \cite{H88} proved the first non-trivial result that
$\lambda(\Gamma)\leq \sqrt{5n-11}$, where $\lambda(\Gamma)$ is the spectral radius of a planar
graph $\Gamma$ on $n\geq3$ vertices. Hong's bound was improved to $4+\sqrt{3n-9}$ by
Cao and Vince \cite{CV93}, and to $2\sqrt{2}+\sqrt{3n-\frac{15}{2}}$ by Hong \cite{H95} himself,
and finally to $2+\sqrt{2n-6}$ by Ellingham and Zha \cite{EZ00}. On the other hand, Boots
and Royle \cite{BR91}, and independently, Cao and Vince \cite{CV93}, conjectured that $P_2\vee P_{n-2}$ attains
the maximum spectral radius among all planar graphs on $n\geq 9$ vertices. Only recently,
Tait and Tobin \cite{TT17} published a proof of the conjecture for sufficiently large graphs.

A graph $G$ is outerplanar if it has a planar embedding $\widetilde{G}$ in which all vertices lie on the
boundary of its outer face. In fact, earlier than the Boots-Royle-Cao-Vince Conjecture,
Cvetkovi\'{c} and Rowlinson \cite{CR90} proposed the following conjecture on outerplanar graphs in
1990. In what follows, $K_1$ denotes a single vertex, $P_{n-1}$ denotes the path on $n-1$ vertices,
and ``$\vee$" is the join operation.

\begin{conj}[Cvetkovi\'{c}, Rowlinson \cite{CR90}] \label{Conj-1}
Among all outerplanar graphs on $n$ vertices,
$K_1\vee P_{n-1}$ attains the maximum spectral radius.
\end{conj}

Cvetkovi\'{c} and Rowlinson \cite{CR90} considered the above conjecture as study on indices of Hamiltonian
graphs. Rowlinson \cite{R90} proved Conjecture \ref{Conj-1} for outerplanar graphs without internal
triangles, where an internal triangle of an outerplanar graph is a 3-cycle which has
no edges in common with the unique Hamiltonian cycle of the graph. For upper bounds of the
spectral radius $\lambda(G)$ of an outerplanar graph $G$, Cao and Vince \cite{CV93} showed that
$\lambda(G)\leq 1+\sqrt{2+\sqrt{2}}+\sqrt{n-5}$. This was improved by Shu and Hong \cite{SH00} to
$\lambda(G)\leq \frac{3}{2}+\sqrt{n-\frac{7}{4}}$. In 2017, Tait and Tobin\cite{TT17}
confirmed Conjecture \ref{Conj-1} for sufficiently large $n$.
\begin{thm}[Tait, Tobin \cite{TT17}]
The Cvetkovi\'{c}-Rowlinson Conjecture is true for all sufficiently large $n$.
\end{thm}
Some variant of the Cvetkovi\'{c}-Rowlinson Conjecture was considered by Yu, Kang, Liu
and Shan \cite{YKLS19}. For related topics on spectral properties of planar graphs, we refer to the
introduction part of \cite{TT17} and references therein.

The humble goal of this article is to give a solution to the Cvetkovi\'{c}-Rowlinson Conjecture
for all $n$. The complete proof consists of two parts. We first prove the conjecture for
$n\geq 17$, and then prove the case that $2\leq n\leq16$ where $n\neq 6$, with the aid of a computer.
We disprove the conjecture for the case of $n=6$.
\begin{thm}\label{Thm-Main}
Among all outerplanar graphs on $n\geq17$ vertices, $K_1\vee P_{n-1}$ attains the
maximum spectral radius.
\end{thm}

Before our proof of Theorem \ref{Thm-Main}, let us introduce some necessary notations and terminology.
Let $G$ be a graph with vertex set $V(G)$ and edge set $E(G)$ and $S\subseteq V (G)$. We
denote by $G[S]$ the subgraph of $G$ induced by $S$ and $G-S$ the subgraph $G[V(G)\backslash S]$.
For any $v\in V(G)$, $N_G(v)$ denotes the set of neighbors of $v$ in $G$, $d_G(v)$ is defined as
$|N_G(v)|$, and $d_S(v):=|N_G(v)\cap S|$. Let $A,B\subset V (G)$ be two disjoint sets. We denote
by $N_A(B):=\bigcup_{v\in B}N_A(v)$, by $d_A(B):=|N_A(B)|$ and by $e_G(A,B)$ the number of edges with one end-vertex
in $A$ and the other one in $B$. If there is no danger of ambiguity, we use $e(A,B)$ instead
of $e_G(A,B)$. Let $G_1$ and $G_2$ be two disjoint graphs. The \emph{join} of $G_1$ and $G_2$,
denoted by $G_1\vee G_2$, is defined as a graph with vertex set $V(G_1)\cup V(G_2)$ and edge set
$E(G_1)\cup E(G_2)\cup \{xy:x\in V(G_1),y\in V(G_2)\}$. Let $A(G)$ be the
adjacency matrix of $G$ and $\lambda(G)$ be the spectral radius of $A(G)$.

A graph $H$ is a \emph{minor} of a graph $G$ if $H$ can be obtained from $G$ by a sequence of vertex
and edge deletions and edge contractions. A complete characterization of outerplanar
graphs states that a graph is outerplanar if and only if it is $K_{2,3}$-minor free and $K_4$-
minor free. It is clear that a subgraph of an outerplanar graph is also outerplanar.
An outerplanar graph is \emph{edge-maximal} (or in short, maximal), if no edge can be added to
the graph without violating outerplanarity. It is well-known that every
outerplanar graph on $n$ vertices has at most $2n-3$ edges if
$n\geq 2$. These properties will be used frequently in our proof.
For some nice article on minors in spectral graph theory, we refer
to \cite{T19}.

Our proof of Theorem \ref{Thm-Main} also needs a well-known fact and an upper bound of
the spectral radius of an outerplanar graph as following:

\begin{lem}[\rm {\cite[Exercise~11.2.7]{BM08}}]\label{Lem:Edge-maximalouterplanar}
Let $G$ be an edge-maximal outerplanar graph of order
$n\geq3$. Then G has a planar embedding whose outer face is a Hamilton cycle, all other
faces being triangles.
\end{lem}

\begin{lem}[Shu, Hong \cite{SH00}]\label{Lem:ShuHong}
Let $G$ be a connected outerplanar graph on $n\geq 3$ vertices. Then
$\lambda(G)\leq\frac{3}{2}+\sqrt{n-\frac{7}{4}}$.
\end{lem}

Now we present a proof of Theorem \ref{Thm-Main}.
\vspace{1cm}

\noindent
{\bf Proof of Theorem \ref{Thm-Main}.}
For any integer $n\geq 17$, let $G_n$ be an outerplanar graph which
attains the maximum spectral radius among all outerplanar graphs of order $n$, and let $\lambda:=\lambda(G_n)$
be its spectral radius. In the rest, we use $G$ instead of $G_n$ for convenience.
Obviously, $G$ is connected and maximal. By the Perron-Frobenius Theorem, $G$ has the Perron vector
such that each component is positive. Let $X$ be a normalized one such that
maximum entry is 1. For any vertex $v\in V(G)$, we write $x_v$ for the eigenvector entry
which corresponds to $v$. Let $u\in V(G)$ such that $x_u=1$, $A=N_G(u)$ and
$B=V(G)-(\{u\}\cup A)$.

The first claim gives us a nearly tight lower bound of $\lambda$.
\begin{claim}\label{Claim:Lowerbound}
$\lambda\geq \sqrt{n}+1-\frac{1}{n-\sqrt{n}}.$
\end{claim}
\begin{proof}
Let $\Gamma=K_1\vee C_{n-1}$, where $C_{n-1}$ denotes a cycle
on $n-1$ vertices. Suppose that $Y=(y_1,y_2,\ldots,y_n)^t$ is the Perron
vector of $\Gamma$, where $y_1$ corresponds to the vertex of degree $n-1$.
By symmetry, $y_2=y_3=\cdots=y_n$. Then $\lambda(\Gamma)y_1=(n-1)y_2$,
$\lambda(\Gamma)y_2=y_1+2y_2$ and $y_1^2+(n-1)y_2^2=1$.
It follows that $\lambda(\Gamma)=1+\sqrt{n}$ and $y^2_2=\frac{1}{2(n-\sqrt{n})}$.
Let $e\in E(C_{n-1})$ and ${\Gamma}'=\Gamma-e$. Then by Rayleigh principle,
$\lambda({\Gamma}')\geq Y^tA({\Gamma}')Y=Y^tA(\Gamma)Y-2y_2^2=\sqrt{n}+1-\frac{1}{n-\sqrt{n}}.$
Obviously, ${\Gamma}'$ is outerplanar, and
$\lambda(G)\geq\lambda({\Gamma}')\geq \sqrt{n}+1-\frac{1}{n-\sqrt{n}},$ as required.
\end{proof}

As a warm up, we quickly determine the structure of $G[A]$ approximately.
\begin{claim}\label{Claim:Approxstructure}
$G[A]$ is a union of disjoint induced paths or an induced path. (In particular, we also view
an isolated vertex in $G[A]$ as an induced path.)
\end{claim}
\begin{proof}
We first claim that $G[A]$ contains no vertex of degree at least $3$ in $A$.
If not, then there is a $K_{2,3}$ in $G[A\cup\{u\}]$, a contradiction.

We then claim that there is no cycle in $G[A]$. Suppose to the contrary that there is
a cycle in $G[A]$. Then we can contract the cycle into a triangle, and there is a $K_4$ in the
resulting graph. That is, there is a $K_4$-minor in $G$, a contradiction.

From the two claims mentioned above, we conclude that $G[A]$ is the union of some
induced paths or an induced paths, in which we view each isolated vertex as an induced
path.
\end{proof}
Let
$$S=\{v: v\in A, d_{G[A]}(v)=1\}.$$
For two vertices $x,y\in V(G)$, we write $x\sim y$ if $x$ is adjacent to $y$. By Claim \ref{Claim:Lowerbound}, we have
$d(u):=d_u\geq \lambda\geq\sqrt{n}+1-\frac{1}{n-\sqrt{n}}>5$.

We want to show that $d_u$ is very close to $n-1$. As a first step, we must associate $d_u$ with
$\lambda$ by the following.
\begin{claim}
\begin{eqnarray}\label{eqn:1}
\lambda^2\leq d_u+2\lambda-\frac{2}{\sqrt{n-\frac{7}{4}}+\frac{3}{2}}+\sum_{v\in B}d_A(v)x_v.
\end{eqnarray}
\end{claim}
\begin{proof}
Note that for any $v\in S$, we have $\lambda x_v>x_u=1$. By Lemma \ref{Lem:ShuHong}, we obtain
$x_v>\frac{1}{\lambda}\geq\frac{1}{\frac{3}{2}+\sqrt{n-\frac{7}{4}}}$.
The first equality below was used by Tait and Tobin (see the proof
of Lemma 4 in \cite{TT17}), which also appeared in several references,
see \cite{FMS93} for example:
\begin{align*}
\lambda^2&=\lambda^2x_u=d_u+\sum_{y\sim u}\sum_{z\in N(y)\cap A}x_z+\sum_{y\sim u}\sum_{z\in N(y)\cap B}x_z=d_u+\sum_{v\in A}d_A(v)x_v+\sum_{v\in B}d_A(v)x_v.
\end{align*}
If $G[A]$ consists of isolated vertices, i.e., without any edge, then $\sum\limits_{v\in A}d_A(v)x_v=0$.
Thus, we have
\begin{align*}
\lambda^2&=d_u+\sum_{v\in B}d_A(v)x_v.
\end{align*}
Otherwise, $G[A]$ contains at least one edge, and it follows $|S|\geq 2$. We have
\begin{align*}
\lambda^2&=d_u+\sum_{v\in A}d_A(v)x_v+\sum_{v\in B}d_A(v)x_v\\
&=d_u+\sum_{v\in S}x_v+\sum_{v\in \{v\in A: d_A(v)=2\}}2x_v+\sum_{v\in B}d_A(v)x_v\\
&\leq d_u+2\lambda-\sum_{v\in S}x_v+\sum_{v\in B}d_A(v)x_v\\
&\leq d_u+2\lambda-\frac{2}{\frac{3}{2}+\sqrt{n-\frac{7}{4}}}+\sum_{v\in B}d_A(v)x_v.
\end{align*}
Since $2\lambda-\frac{2}{\frac{3}{2}+\sqrt{n-\frac{7}{4}}}>0$ for $n\geq 16$,
we have proved the claim.
\end{proof}

Our goal of the most of the rest is to show that $|B|\leq 1$ firstly, and then show $B=\emptyset$.
We prove this fact by contradiction. Suppose to the contrary that
\begin{align}\label{align:2}
|B|\geq 2.
\end{align}
Since $G$ is outerplanar, $G[B]$ is also outerplanar, and so $e(G[B])\leq 2|B|-3$ by (\ref{align:2}).
In the rest, let $B_1, B_2,\ldots,B_t$ be the vertex sets of all components of $G[B]$, respectively.
The coming claim gives a tight upper bound of the sum of all degrees of vertices of $B$
in $G$, which plays a central role in our proof. Since adding a new edge can increase the
value of the spectral radius (recall $G$ is connected), $G$ is a maximal outerplanar
graph. Therefore, Lemma \ref{Lem:Edge-maximalouterplanar} can be used below.

\begin{claim}\label{Claim:MainOne}
(i) For each $i\in [1,t]$, $d_A(B_i)=2$.
(ii) If $|B_i|\geq 2$, then $2e(G[B_i])+e(A,B_i)\leq 4|B_i|-3$.
In particular, $2e(G[B])+e(A,B)\leq 4|B|-3$ (recall $|B|\geq 2$).
\end{claim}
\begin{proof}
(i) Since $G$ is $K_{2,3}$-minor free, $B_i$ has at most 2 neighbors in $A$
for any $i\in [1,t]$. Indeed, if not, we contract all vertices of $B_i$ into
a single vertex, and would find a $K_{2,3}$ in the resulting graph.
Thus, $d_A(B_i)\leq 2$. Recall that there is a Hamilton
cycle in $G$. Thus, $d_A(B_i)=2$. This proves Claim \ref{Claim:MainOne}~(i).

(ii) By Claim \ref{Claim:MainOne}~(i), we can assume that $N_A(B_i)=\{x,x'\}$ for
any $i\in [1,t]$. By Lemma \ref{Lem:Edge-maximalouterplanar}, there
is a planar embedding of $G$, say $\widetilde{G}$, such that its outer-face is a Hamilton cycle. Let $P:=xp_1p_2\cdots p_sx'$ be
the $(x,x')$-path on the Hamilton cycle passing through all vertices in $B_i$. That is, $B_i=\{p_1,\ldots,p_s\}$.
In the rest of the proof, when there is no danger of ambiguity, we do not
distinguish $G$ and $\widetilde{G}$.

Suppose that $|B_i|\geq 2$. We first claim that there are no subscripts $j,k$ such that $1\leq j<k\leq s$ and $xp_k,x'p_j\in E(G)$.
Suppose not. Then we first contract three paths $p_1\ldots p_j$, $p_k\ldots p_s$ and $xux'$
into vertices $w_1,w_2$ and an edge $xx'$, respectively, and then
contract the path $w_1p_{j+1}\ldots p_{k-1}w_2$ into an edge $w_1w_2$,
resulting in a $K_4$. In this way, we can find a $K_4$-minor in $G$,
a contradiction. In the following, set $l_1:=\max\{q:p_qx\in E(G)\}$
and $l_2:=\min\{q:p_qx'\in E(G)\}$. Therefore $l_1\leq l_2$.
Also, $G_1:=G[\{x,p_1,\ldots,p_{l_1}\}]$ is outerplanar, and hence $e(G_1)\leq 2(l_1+1)-3=2l_1-1$. Note that $G_2:=G[p_{l_1},\ldots,p_{l_2},x,x']-xx'$ is outerplanar. Thus, if $l_2\geq l_1+1$,
then $e(G_2)\leq e(G[\{p_{l_1},\ldots,p_{l_2}\}])+2\leq 2(l_2-l_1+1)-3+2=2(l_2-l_1)+1$; if $l_1=l_2$
then $e(G_2)=2$. Let $G_3:=G[\{p_{l_2},\ldots,p_s,x'\}]$. Then $e(G_3)\leq 2(s-l_2+1+1)-3=2(s-l_2)+1$.

Observe that for any $i\in [1,l_1]$ and $j\in [l_2,s]$ such that $j\geq i+2$, we have $p_ip_j\notin E(G)$,
since otherwise we can find a $K_4$-minor in $G$ similarly as above. Hence $e(G[B_i\cup \{x,x'\}]-xx')=e(G_1)+e(G_2)+e(G_3)-2$,
where the term ``-2" comes from the fact that
the edges $xp_{l_1},x'p_{l_2}$ are counting twice when we compute the value of $e(G_1)+e(G_2)+e(G_3)$.

If $l_2\geq l_1+1$, then $e(G[B_i\cup \{x,x'\}]-xx')=e(G_1)+e(G_2)+e(G_3)-2\leq (2l_1-1)+(2(l_2-l_1)+1)+(2(s-l_2)+1)-2=2s-1$, Thus,
$2e(G[B_i])+e(A,B_i)\leq 2e(G[B_i\cup \{x,x'\}]-xx')-e(A,B_i)\leq 2(2s-1)-3=4s-5$,
where $e(A,B_i)\geq 3$ since $|B_i|\geq2$ and each face inside $\widetilde{G}[B_i\cup \{x,x'\}]$
is a triangle.

If $l_2=l_1$, then $e(G_2)=2$. In this case, $e(G[B_i\cup \{x,x'\}]-xx')=e(G_1)+e(G_2)+e(G_3)-2\leq (2l_1-1)+2+(2(s-l_2)+1)-2=2s$.
Then $2e(G[B_i])+e(A,B_i)\leq 2e(G[B_i\cup \{x,x'\}]-xx')-e(A,B_i)\leq 2\cdot(2s)-3=4s-3$, where $e(A,B_i)\geq 3$
since $|B_i|\geq 2$.

Thus, for any $i\in[1,t]$ with $|B_i|\geq 2$, we have $2e(G[B_i])+e(A,B_i)\leq 4s-3$. If $|B_i|=1$
then $2e(G[B_i])+e(A,B_i)\leq 2$. Summing over all indices $i$,
we have $e(B,A)+2e(G[B])\leq 4|B|-3$. This proves Claim \ref{Claim:MainOne}~(ii).
\end{proof}
By using Claim \ref{Claim:MainOne}~(ii), we can estimate the upper bound of $\sum\limits_{v\in B}d_A(v)x_v$
as follows.
\begin{claim}
\begin{eqnarray}\label{eqn:3}
\sum_{v\in B}d_A(v)x_v\leq \frac{5n-5d_u-7}{\sqrt{n}+1-\frac{1}{n-\sqrt{n}}}.
\end{eqnarray}
\end{claim}
\begin{proof}
Recall that $B_1,B_2,\ldots,B_t$ are all components of $G[B]$.
For any $i\in[1,t]$, by Claim \ref{Claim:MainOne}~(i), $B_i$ has two neighbors in $A$.
Since $G$ contains no $K_{2,3}$, there is at most one vertex in $B_i$ with two
neighbors in $A$. Set $x'_i:=\max\{x_v:v\in B_i\}$. Thus, if $|B_i|\geq 2$ then
\begin{align*}
\sum_{v\in B_i}d_A(v)x_v&\leq \sum_{v\in B_i}x_v+x'_i=\frac{1}{\lambda}(\sum_{v\in B_i}\lambda x_v+\lambda x'_i)\\
&\leq \frac{1}{\lambda}(\sum_{v\in B_i}d_G(v)+(|B_i|-1+2))\\
&=\frac{1}{\lambda}\left(e(A,B_i)+2e(G[B_i])+|B_i|+1\right)\\
&=\frac{1}{\lambda}(5|B_i|-2).
\end{align*}
If $|B_i|=1$ then $\sum\limits_{v\in B_i}d_A(v)x_v\leq \frac{2}{\lambda}\sum\limits_{w\in N_A(B_i)}x_w\leq \frac{4}{\lambda}=\frac{1}{\lambda}(5|B_i|-1)$. Observe that if
$|B_i|=1$ for every $i$, then $t\geq 2$ since $|B|\geq2$.
Summing over all $i\in [1,t]$, we have
$$\sum_{v\in B}d_A(v)x_v\leq \frac{5|B|-2}{\sqrt{n}+1-\frac{1}{n-\sqrt{n}}}= \frac{5n-5d_u-7}{\sqrt{n}+1-\frac{1}{n-\sqrt{n}}}.$$
This proves the claim.
\end{proof}
In what follows, we aim to show that
\begin{align}
\left(1-\frac{5}{\sqrt{n}+1-\frac{1}{n-\sqrt{n}}}\right)d_u>\max\left\{(n-1)\cdot \left(1-\frac{5}{\sqrt{n}+1-\frac{1}{n-\sqrt{n}}}\right),0\right\}.
\end{align}
holds for $n\geq 17$. This finally results in $d(u)>n-1$, and implies that $|B|\geq 2$ does not hold.

By (\ref{eqn:1}), (\ref{align:2}) and (\ref{eqn:3}), we infer
\begin{align*}
&\left(1-\frac{5}{\sqrt{n}+1-\frac{1}{n-\sqrt{n}}}\right)d_u\\
&\geq {\lambda}^2-2\lambda+\frac{2}{\frac{3}{2}+\sqrt{n-\frac{7}{4}}}-\frac{5n-7}{\sqrt{n}+1-\frac{1}{n-\sqrt{n}}}\\
&\geq n-1-\frac{2}{\sqrt{n}-1}+\frac{1}{n(\sqrt{n}-1)^2}+\frac{2}{\frac{3}{2}+\sqrt{n-\frac{7}{4}}}-\frac{5n-7}{\sqrt{n}+1-\frac{1}{n-\sqrt{n}}}\\
&>(n-1)\cdot\left(1-\frac{5}{\sqrt{n}+1-\frac{1}{n-\sqrt{n}}}\right)-\frac{2}{\sqrt{n}-1}+\frac{2}{\frac{3}{2}+\sqrt{n-\frac{7}{4}}}+\frac{2}{\sqrt{n}+1}\\
&>(n-1)\cdot\left(1-\frac{5}{\sqrt{n}+1-\frac{1}{n-\sqrt{n}}}\right)+\frac{2\sqrt{n-\frac{7}{4}}-7}{n-1}\\
&\geq (n-1)\cdot\left(1-\frac{5}{\sqrt{n}+1-\frac{1}{n-\sqrt{n}}}\right)>0
\end{align*}
for $n\geq 17$. (Note that $(1-\frac{5}{\sqrt{n}+1-\frac{1}{n-\sqrt{n}}})<0$ when $n=16$.)

Therefore, we have $|B|\leq1$. Suppose that $|B|=1$. At this point, we can know more
information on $G[A]$ than Claim \ref{Claim:Approxstructure}.
\begin{claim}
$G[A]$ is an induced path.
\end{claim}
\begin{proof}
By Claim \ref{Claim:Approxstructure}, $G[A]$ is a union of disjoint induced paths or an induced path. Since
$G$ is a maximal outerplanar graph, by Lemma \ref{Lem:Edge-maximalouterplanar},
$G$ has a planar embedding, say $\widetilde{G}$,
whose outer face is a Hamilton cycle, all other faces being triangles. If $G[A]$ is not an
induced path, then the fact $|B|=1$ implies there is an inner face in $\widetilde{G}$
which is not a triangle, a contradiction. This proves the claim.
\end{proof}

Finally, we show that, indeed, $B$ is an empty set.

\begin{claim}
$B=\emptyset.$
\end{claim}
\begin{proof}
Suppose that $|B|=1$. Let $B=\{v\}$. Since $G$ is $K_{2,3}$-free, we
have $d(v)=2$. Set $N(v)=\{v_i,v_j\}$. Recall that $G[A]$ is an induced path. Let $G[A]=
v_1v_2\ldots v_{n-2}$. If $|i-j|\neq 1$, then $G$ contains a $K_{2,3}$-minor, a contradiction. Thus, $|i-j|=1$.
Without loss of generality, set $j=i+1$. Since $d_u>5$, at least one vertex of $\{v_i,v_{i+1}\}$
has degree two in $G[A]$. Moreover, $vv_i,vv_{i+1}\in E(G)$.
Let $X=(x_u,x_1,\ldots,x_{n-2},x_v)^t$ be the eigenvector corresponding to $\lambda(G)$, where $x_u=1$,
$x_v$ corresponds to $v$, and $x_k$ corresponds to $v_k$ for $k=1,\ldots,n-2$.
Set $x_s:=\max\{x_k:k=1,\ldots,n-2\}$. By Claim \ref{Claim:Lowerbound},
$\lambda\geq \sqrt{n}+1-\frac{1}{n-\sqrt{n}}\geq 5.$ Then $\lambda x_v=x_i+x_{i+1}\leq 2x_s$, which implies that $x_v<x_s.$
Since $\lambda x_{s}\leq x_u+\sum\limits_{v_k\sim v_{s}}x_k+x_v<1+3x_s$, it follows that $x_{s}<\frac{1}{\lambda-3}.$
Also, since $\lambda x_1>x_u$, we have $x_1>\frac{1}{\lambda}$.

Now let $G':=G-vv_i-vv_{i+1}+vu+vv_1$. Note that $G'$ is also outerplanar and
$\lambda(G')-\lambda(G)\geq 2X^t(A(G')-A(G))X=2x_v(x_u+x_{1}-x_i-x_{i+1})>2x_v(1+\frac{1}{\lambda}-\frac{2}{\lambda-3})$.
By simple algebra, $1+\frac{1}{\lambda}-\frac{2}{\lambda-3}=\frac{\lambda^2-4\lambda-3}{\lambda(\lambda-3)}$.
In order to prove the inequality $\frac{\lambda^2-4\lambda-3}{\lambda(\lambda-3)}>0$,
it suffices to show $\lambda>2+\sqrt{7}$. By Claim \ref{Claim:Lowerbound},$\lambda>2+\sqrt{7}$
when $n\geq 16$. Therefore $\lambda(G')>\lambda(G)$, a contradiction. This
proves the claim.
\end{proof}
It follows that $G=K_1\vee P_{n-1}$, completing the proof.  $\hfill\blacksquare$

\vspace{0.2cm}
\noindent
{\bf The case of $2\leq n\leq16$.}
\vspace{0.2cm}

Throughout this part, we use the notation NIM-outerplanar graphs instead of non-isomorphic
maximal outerplanar graphs.

Let $G$ be a maximal outerplanar graph with order $n$ and vertex set $V(G)=\{v_1,\ldots,v_n\}$.
By Lemma \ref{Lem:Edge-maximalouterplanar}, $G$ has a planar embedding, say $\widetilde{G}$,
such that the outer-face is a Hamilton cycle. One can easily find such a Hamilton cycle
is unique, since otherwise there is a $K_4$-minor in $G$. In the following, we do not
distinguish a maximal outerplanar graph and its planar embedding when there is no
ambiguity. Let $v_1\ldots v_nv_1$ be the Hamilton cycle mentioned above.

A property of a maximal outerplanar graph on $n\geq 3$ vertices states that every such graph has a vertex
of degree 2 and has a subgraph which is also maximal outerplanar by deleting the vertex.
Let $G'$ be a maximal outerplanar graph obtained from $G$ by adding a new vertex $v_{n+1}$,
which is adjacent to some vertex(vertices) of $G$. Since
$e(G')=2(n+1)-3=e(G)+2$, we have $d_{G'}(v_{n+1}) = 2$. Furthermore,
$N_{G'}(v_{n+1})=\{v_i,v_{i+1}\}$ for some $i\in [1,n]$. Let $S(n)$ denote the
number of NIM-outerplanar graphs with order $n$. It follows that $S(n+1)\leq nS(n)$. In
fact, by using matlab, the number of NIM-outerplanar graphs with order at most 13 can
be computed as shown in Table 1. In particular, $S(14)\leq 29666$, $S(15)\leq 415324$
and $S(16)\leq6229860$.

\begin{table}[H]
\begin{tabular}{|c|c|c|c|c|c|c|c|c|c|c|c|}
\hline$n$ & 6 & 7 & 8 & 9 & 10 & 11 & 12 & 13 & 14 & 15 & 16 \\
\hline$S(n)$ & 3 & 4 & 12 & 27 & 82 & 228 & 733 & 2282 & $\leq 29666$ & $\leq 415324$ & $\leq 6229860$ \\
\hline
\end{tabular}
\caption{The numbers of NIM-outerplanar graphs with order at least 6 and at most 16.}
\label{tab-1}
\end{table}

For the case of $2\leq n\leq5$, one can check by hand calculation. For the case of
$6\leq n\leq13$, the problem can be completely solved by a computer. When $14\leq n\leq16$,
we first develop a program to determine all NIM-outerplanar graphs of order 13, and then
compare the spectral radius of each graph in the family of no more than
$29666~(415324,6229860)$ graphs and of $K_1\vee P_{n-1}$, respectively (see https://github.com/HuiqiuLin-83/Outerplanar-graph/ for the code). Furthermore,
we find out $\lambda(G_1)=3.2361>\lambda(K_1\vee P_5)=3.2227$.
In summary, we get the following result (together with Theorem \ref{Thm-Main}).

\begin{thm}
Among all outerplanar graphs on $n$ vertices, $K_1\vee P_{n-1}$ attains the maximum
spectral radius, with the only exceptional case of $n=6$, in which $G_1$ attains the maximum
spectral radius. (see Figure 1).
\end{thm}
\begin{figure}[H]
\centering
\includegraphics[width=4cm]{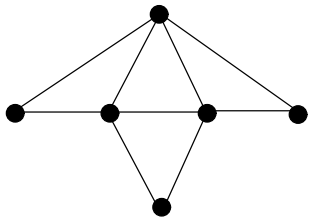}
\caption{The graph $G_1$.}
\label{fig1}
\end{figure}
\section*{Concluding remarks}

We would like to compare our proof with the methods applied by Tait and Tobin in \cite{TT17}.
Generality speaking, the Tait-Tobin Method is somewhat motivated by Regularity Lemma.
For any $\varepsilon>0$, they partitioned the vertex set $V(G)$ into two parts $V_L=\{v\in V(G): x_v>\varepsilon\}$
and $V_S=V(G)\backslash V_L$. ``This enables them to determine the structure on large (linear
sized) pieces of the graph to understand approximately what the extremal graph looks
like."\footnote{This sentence is essentially borrowed from an email from M. Tait to the authors,
who explained the ideas on the Tait-Tobin Method to them.} Our proof heavily relies
on the complete characterization of outerplanar graphs, and
the concept of ``minor" plays an important role in our proof. More importantly for us, we
need a Hamilton cycle in a maximal outerplanar graph to label the vertices in order. We
need much more details (see Claims 4 and 5) to control the bound of eigenvector entries,
besides using Shu-Hong's inequality.

On the other hand, the Tait-Tobin Method seems to be powerful for many problems
(on large graphs) in spectral graph theory. Till now, it has been successfully used to make progress
on Cvetkovi\'{c}-Rowlinson Conjecture \cite{CR90}, Boots-Royle-Cao-Vince Conjecture \cite{BR91,CV93},
and Cioab\u{a}-Gregory Conjecture \cite{CG07}, etc. We would like to refer to the very recent
spectral version \cite{CFTZ20} of extremal numbers of friendship graphs \cite{EFGG95}.

In closing, we shall mention the following conjecture again, which is still open for
small $n$. (For example, let us confirm this conjecture for all $n\geq 100$.) We note that Tait
and Tobin \cite{TT17} verified it for sufficiently large $n$.
\begin{conj}[Boots-Royle \cite{BR91}, and independently by Cao-Vince \cite{CV93}]
The planar graph on $n\geq 9$ vertices of maximum spectral radius is $P_2\vee P_{n-2}$.
\end{conj}

\section*{Acknowledgment}
The authors are grateful to Jun Ge for helpful comments on the original draft, and to
Xueyi Huang for providing help for the code. They are also grateful to Michael Tait for
discussions on the Tait-Tobin Method.

\end{document}